\documentclass[11pt]{article}

\usepackage[a4paper,margin=28mm]{geometry}

\usepackage{amsmath,amssymb,amsthm,mathtools}
\usepackage{mathrsfs}

\usepackage{newtxtext,newtxmath}
\usepackage{bm}

\usepackage{graphicx}
\usepackage{xcolor}
\usepackage{tikz}
\usetikzlibrary{arrows.meta}

\usepackage{hyperref}
\hypersetup{
	colorlinks=true,
	linkcolor=blue,
	citecolor=blue,
	urlcolor=blue
}

\theoremstyle{plain}
\newtheorem{theorem}{Theorem}[section]

\newtheorem{proposition}[theorem]{Proposition}
\newtheorem{corollary}[theorem]{Corollary}

\theoremstyle{definition}
\newtheorem{definition}[theorem]{Definition}
\newtheorem{example}[theorem]{Example}

\theoremstyle{remark}
\newtheorem{remark}[theorem]{Remark}

\title{\Large Generalized Frenet frames and frame sequences of singular space curves}
\author{Shun'ichi Honda}
\date{\today}

\begin{document}
	
	\maketitle
	
	\begin{abstract}
	The classical Frenet frame is defined by a concrete construction from the tangent, principal normal, and binormal vectors of a regular space curve. 
	However, this construction breaks down at singular points and at points where the curvature vanishes. 
	Motivated by this observation, we reconsider the Frenet frame from an axiomatic viewpoint and identify the fundamental properties that characterize it independently of its classical construction. 
	Based on the theory of frontals on the unit sphere and Legendre duality, we introduce a generalized Frenet frame for singular space curves. 
	Furthermore, we introduce the notion of a frame sequence, which gives rise to a family of $k$th-order Frenet frames and their corresponding curvatures and torsions, indexed by $k \in \mathbb{Z}$. 
	This viewpoint provides a unified framework encompassing the Frenet and Bishop frames of space curves as well as the evolute-involute correspondence for spherical frontals. 
	Moreover, explicit recursive formulas are derived, revealing that the curvatures and torsions at each level encode, respectively, the magnitude and rotational behavior of the invariants at the preceding level. 
	\end{abstract}
	
	\medskip
	
	
	\medskip
	
	
	\section{Introduction}
	The classical Frenet frame is one of the fundamental tools in the differential geometry of space curves. 
	For a regular curve with non-vanishing curvature, it provides an orthonormal frame consisting of the tangent, principal normal, and binormal vectors. 
	The curvature and torsion, regarded as functions of arc length, form a complete system of invariants under orientation-preserving Euclidean motions. 
	However, the classical construction breaks down at singular points and at points where the curvature vanishes. 
	This observation suggests that the classical construction may not fully reveal the essential geometric structure of the Frenet frame. 
	It is therefore natural to ask which geometric properties characterize the Frenet frame independently of its classical construction. 
	
	In this paper, we reconsider the Frenet frame from an axiomatic viewpoint. 
	Rather than starting from the classical formulas, we focus on the geometric structure underlying the Frenet frame and identify the fundamental properties that characterize it. 
	This viewpoint leads naturally to the notion of a generalized Frenet frame based on spherical frontals and Legendre duality. 
	We prove that every classical Frenet curve admits a generalized Frenet frame, while the definition also applies to a broad class of singular curves. 
	Furthermore, we introduce the notion of a frame sequence, which gives rise to a family of $k$th-order Frenet frames and their corresponding curvatures and torsions, indexed by $k \in \mathbb{Z}$. 
	This framework provides a unified viewpoint on several classical constructions in the differential geometry of curves. 
	In particular, the classical Frenet frame and the Bishop frame correspond to the $0$th-order and $(-1)$st-order Frenet frames, respectively. 
	Moreover, the frame sequence admits a natural interpretation in terms of the evolute-involute correspondence for spherical frontals. 
	We also derive explicit recursive formulas relating the curvatures and torsions at consecutive levels. 
	These formulas reveal that the curvatures and torsions at each level encode, respectively, the magnitude and rotational behavior of the invariants at the preceding level. 
	
	The paper is organized as follows. 
	In Section 2, we review spherical Legendre curves, Frenet curves, and framed curves. 
	In Section 3, we introduce generalized Frenet frames and establish their basic properties. 
	In Section 4, we develop the theory of frame sequences and derive recursive formulas for the associated curvatures and torsions. 	
	
	Throughout this paper, all curves, mappings, and manifolds are assumed to be of class $C^\infty$ unless otherwise stated. 
	
	\section{Preliminaries}
	Let $\mathbb{R}^3$ be the three-dimensional Euclidean space equipped with the standard inner product $\langle \bm{a}, \bm{b} \rangle = a_1 b_1 + a_2 b_2 + a_3 b_3$, where $\bm{a} = (a_1, a_2, a_3)$ and $\bm{b} = (b_1, b_2, b_3)$. 
	The norm (or length) of a vector $\bm{a} \in \mathbb{R}^3$ is defined by $\| \bm{a} \| = \sqrt{\langle \bm{a}, \bm{a} \rangle}$. 
	The cross product of two vectors $\bm{a}, \bm{b} \in \mathbb{R}^3$ is defined by 
	\[
	\bm{a} \times \bm{b} = 
	\begin{vmatrix}
		a_1 & a_2 & a_3 \\
		b_1 & b_2 & b_3 \\
		\bm{e}_1 & \bm{e}_2 & \bm{e}_3
	\end{vmatrix}
	= (a_2 b_3 - a_3 b_2, a_3 b_1 - a_1 b_3, a_1 b_2 - a_2 b_1), 
	\]
	where $\{ \bm{e}_1, \bm{e}_2, \bm{e}_3 \}$ is the standard basis of $\mathbb{R}^3$. 
	
	\subsection{Spherical Legendre curves}
	We recall the theory of frontals on the unit sphere $S^2 \subset \mathbb{R}^3$, which plays a central role in this paper. 
	For a detailed exposition, we refer to \cite{T1}. 
	We denote by $\Delta$ the set of orthogonal pairs of unit vectors, that is, $\Delta = \{ (\bm{a}, \bm{b}) \in S^2 \times S^2 \mid \langle \bm{a}, \bm{b} \rangle = 0 \}$. 
	Let $I \subset \mathbb{R}$ be an interval. 
	
	\begin{definition}\label{leg.curve}
		We say that a map $(\bm{\gamma}, \bm{\nu}):I \to \Delta$ is a (spherical) \textit{Legendre curve} if $\langle \dot{\bm{\gamma}}(t), \bm{\nu}(t) \rangle = 0$ for all $t \in I$. 
		A curve $\bm{\gamma}:I \to S^2$ is called a (spherical) \textit{frontal} if there exists a map $\bm{\nu}:I \to S^2$ such that $(\bm{\gamma}, \bm{\nu})$ is a Legendre curve in $\Delta$. 
		When $(\bm{\gamma}, \bm{\nu})$ is a Legendre curve, we call $\bm{\nu}$ a \textit{dual} of the frontal $\bm{\gamma}$. 
	\end{definition}
	
	Let $(\bm{\gamma}, \bm{\nu}):I \to \Delta$ be a Legendre curve. 
	We define $\bm{\mu}(t) = \bm{\gamma}(t) \times \bm{\nu}(t)$. 
	Then $\{ \bm{\gamma}, \bm{\nu}, \bm{\mu} \}$ forms an orthonormal frame along $\bm{\gamma}$. 
	
	\begin{proposition}[\cite{T1}]\label{spherical.FS}
		Let $(\bm{\gamma}, \bm{\nu}):I \to \Delta$ be a Legendre curve. 
		Then the following Frenet-Serret type formulas hold: 
		\begin{align}\label{FS1}
			\frac{d}{dt}
			\begin{pmatrix}
				\bm{\gamma}(t) \\
				\bm{\nu}(t) \\
				\bm{\mu}(t)
			\end{pmatrix}
			=
			\begin{pmatrix}
				0 & 0 & m(t) \\
				0 & 0 & n(t) \\
				-m(t) & -n(t) & 0
			\end{pmatrix}
			\begin{pmatrix}
				\bm{\gamma}(t) \\
				\bm{\nu}(t) \\
				\bm{\mu}(t)
			\end{pmatrix}, 
		\end{align}
		where $m(t) = \langle \dot{\bm{\gamma}}(t), \bm{\mu}(t) \rangle$ and $n(t) = \langle \dot{\bm{\nu}}(t), \bm{\mu}(t) \rangle$. 
	\end{proposition}
	
	We call $(m, n)$ the \textit{curvature} of the Legendre curve $(\bm{\gamma}, \bm{\nu})$. 
	The existence and uniqueness theorem for Legendre curves with prescribed curvature $(m,n)$ was established in \cite{T1}. 
	
	\begin{remark}
		Let $(\bm{\gamma}, \bm{\nu}): I \to \Delta$ be a Legendre curve. 
		Then $(\bm{\nu}, \bm{\gamma}): I \to \Delta$ is also a Legendre curve. 
		This is one reason why $\bm{\gamma}$ and $\bm{\nu}$ are called dual to each other. 
		Indeed, the Frenet-Serret type formulas reveal that $\bm{\gamma}$ and $\bm{\nu}$ play symmetric roles.
	\end{remark}
	
	\begin{remark}\label{rem.analytic.frontal}
		It was shown in \cite{T1} that every real-analytic curve germ $\bm{\gamma}:(I,t_0) \to S^2$ is a spherical frontal. 
		Therefore, for real-analytic spherical curves, the theory of spherical Legendre curves can always be applied locally. 
	\end{remark}
	
	We summarize the notions of evolute and involute for frontals in $S^2$, since they play an important role in understanding the properties of the Frenet frame and the Bishop frame discussed later. 
	For detailed accounts of evolutes and involutes of frontals in $S^2$, we refer the reader to \cite{LP1, T1}. 
		
	\begin{definition}
		Let $(\bm{\gamma}, \bm{\nu}): I \to \Delta$ be a Legendre curve, and let 
		$\bm{\mu}(t) = \bm{\gamma}(t) \times \bm{\nu}(t)$. 
		If there exists a map $\bm{\eta}: I \to S^2$ such that $(\bm{\mu}, \bm{\eta}): I \to \Delta$ is a Legendre curve, 
		then $\bm{\eta}$ is called an \textit{evolute} of $(\bm{\gamma}, \bm{\nu})$. 
	\end{definition}
	
	Geometrically, the evolute corresponds to the dual curve associated with the (generalized) tangent vector of $\bm{\gamma}$, which is given by $\bm{\mu}(t) = \bm{\gamma}(t) \times \bm{\nu}(t)$. 
	Furthermore, by the Frenet-Serret type formulas (see \eqref{FS1}), the evolutes of $\bm{\gamma}$ and $\bm{\nu}$ coincide. 
	These relationships are summarized in Figure~\ref{fig:evolute}. 

	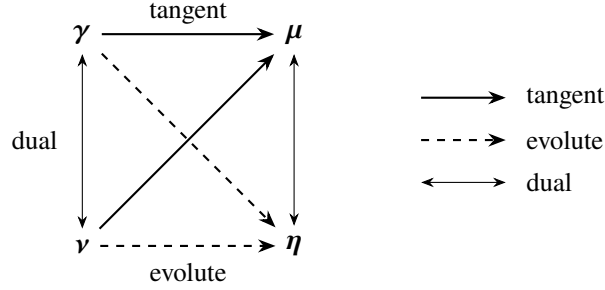
\begin{figure}[htbp]
		\centering
		\begin{tikzpicture}[
			>=Stealth,
			scale=1.4,
			tangent/.style={->, thick},
			evolute/.style={->, dashed, thick},
			dual/.style={<->, thin}
			]
			
			\node (gamma) at (0,2) {$\bm{\gamma}$};
			\node (nu)    at (0,0) {$\bm{\nu}$};
			\node (mu)    at (2,2) {$\bm{\mu}$};
			\node (eta)   at (2,0) {$\bm{\eta}$};
			
			\draw[tangent] (gamma) -- (mu);
			\draw[tangent] (nu) -- (mu);
			\draw[dual]    (gamma) -- (nu);
			\draw[dual]    (mu) -- (eta);
			\draw[evolute] (gamma) -- (eta);
			\draw[evolute] (nu) -- (eta);
			
			\node at (1,2.2) {\small tangent};
			\node at (-0.45,1) {\small dual}; 
			\node at (1,-0.25) {\small evolute};
			
			\draw[tangent] (3.2,1.4) -- (4.0,1.4);
			\node[anchor=west] at (4.1,1.4) {\small tangent};
			
			\draw[evolute] (3.2,1.0) -- (4.0,1.0);
			\node[anchor=west] at (4.1,1.0) {\small evolute};
			
			\draw[dual] (3.2,0.6) -- (4.0,0.6);
			\node[anchor=west] at (4.1,0.6) {\small dual};
			
		\end{tikzpicture}
		\caption{
			Relationships among $\bm{\gamma}$, $\bm{\nu}$, $\bm{\mu}$, and $\bm{\eta}$; 
			the evolutes of $\bm{\gamma}$ and $\bm{\nu}$ coincide. 
		}
		\label{fig:evolute}
	\end{figure}
	
	\begin{definition}
		Let $(\bm{\gamma}, \bm{\nu}): I \to \Delta$ and $(\bm{\sigma}, \bm{\tau}): I \to \Delta$ be Legendre curves. 
		$\bm{\sigma}$ is called an \textit{involute} of $\bm{\nu}$ if $\bm{\nu}$ is an evolute of $(\bm{\sigma}, \bm{\tau})$. 
	\end{definition}
	
	Note that if $\bm{\sigma}$ is an involute of $\bm{\nu}$, then $\bm{\tau}$ is also an involute of $\bm{\nu}$. 
	The correspondences among $\bm{\sigma}$, $\bm{\tau}$, $\bm{\gamma}$, and $\bm{\nu}$ are summarized in Figure~\ref{fig:involute}. 
	Moreover, there exist infinitely many involutes of $\bm{\nu}$. 
	Further details on this non-uniqueness will be discussed later. 
	
	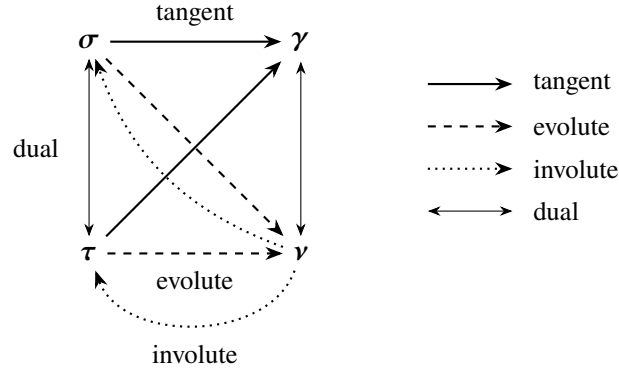
\begin{figure}[htbp]
		\centering
		\begin{tikzpicture}[
			>=Stealth,
			scale=1.4,
			tangent/.style={->, thick},
			evolute/.style={->, dashed, thick},
			involute/.style={->, dotted, thick},
			dual/.style={<->, thin}
			]
			
			\node (sigma) at (0,2) {$\bm{\sigma}$};
			\node (tau)   at (0,0) {$\bm{\tau}$};
			\node (gamma) at (2,2) {$\bm{\gamma}$};
			\node (nu)    at (2,0) {$\bm{\nu}$};
			
			\draw[tangent] (sigma) -- (gamma);
			\draw[tangent] (tau) -- (gamma);
			
			\draw[dual] (sigma) -- (tau);
			\draw[dual] (gamma) -- (nu);
			
			\draw[evolute] (sigma) -- (nu);
			\draw[evolute] (tau) -- (nu);
			
			\draw[involute, bend left=25] (nu) to (sigma);
			\draw[involute] (nu) to[out=250, in=290] (tau);
			
			\node at (1,2.25) {\small tangent};
			\node at (-0.5,1) {\small dual};
			\node at (1,-0.25) {\small evolute};
			\node at (1,-0.95) {\small involute};
			
			\draw[tangent] (3.2,1.6) -- (4.0,1.6);
			\node[anchor=west] at (4.1,1.6) {\small tangent};
			
			\draw[evolute] (3.2,1.2) -- (4.0,1.2);
			\node[anchor=west] at (4.1,1.2) {\small evolute};
			
			\draw[involute] (3.2,0.8) -- (4.0,0.8);
			\node[anchor=west] at (4.1,0.8) {\small involute};
			
			\draw[dual] (3.2,0.4) -- (4.0,0.4);
			\node[anchor=west] at (4.1,0.4) {\small dual};
			
		\end{tikzpicture}
		
		\caption{
			Relationships among $\bm{\sigma}$, $\bm{\tau}$, $\bm{\gamma}$, and $\bm{\nu}$; 
			$\bm{\sigma}$ and $\bm{\tau}$ are involutes of $\bm{\nu}$. 
		}
		\label{fig:involute}
	\end{figure}
	
	\subsection{Frenet curves} \label{sec.frenet}
	We review basic concepts in the classical differential geometry of regular space curves in $\mathbb{R}^3$. 
	Suppose that $\bm{\gamma} : I \to \mathbb{R}^3$ is a regular space curve such that $\dot{\bm{\gamma}}(t)$ and $\ddot{\bm{\gamma}}(t)$ are linearly independent for all $t \in I$. 
	Such a regular curve satisfying this linear independence condition is called a \textit{Frenet curve}. 
	Then we obtain an orthonormal frame
	\[
	\left\{
	\bm{t}(t), \bm{n}(t), \bm{b}(t)
	\right\}
	=
	\left\{
	\frac{\dot{\bm{\gamma}}(t)}{\| \dot{\bm{\gamma}}(t) \|}, \frac{(\dot{\bm{\gamma}}(t) \times \ddot{\bm{\gamma}}(t)) \times \dot{\bm{\gamma}}(t)}{\| (\dot{\bm{\gamma}}(t) \times \ddot{\bm{\gamma}}(t)) \times \dot{\bm{\gamma}}(t) \|}, \frac{\dot{\bm{\gamma}}(t) \times \ddot{\bm{\gamma}}(t)}{\| \dot{\bm{\gamma}}(t) \times \ddot{\bm{\gamma}}(t) \|}
	\right\}
	\]
	along $\bm{\gamma}$, called the Frenet frame. 
	Here, $\bm{t}$, $\bm{n}$, and $\bm{b}$ are referred to as the \textit{tangent}, \textit{principal normal}, and \textit{binormal vectors}, respectively. 
	Then the following Frenet-Serret formulas hold: 
	\[
	\frac{d}{dt}
	\begin{pmatrix}
		\bm{t}(t) \\
		\bm{n}(t) \\
		\bm{b}(t)
	\end{pmatrix}
	=
	\| \dot{\bm{\gamma}}(t) \|
	\begin{pmatrix}
		0 & \kappa(t) & 0 \\
		- \kappa(t) & 0 & \tau(t) \\
		0 & -\tau(t) & 0
	\end{pmatrix}
	\begin{pmatrix}
		\bm{t}(t) \\
		\bm{n}(t) \\
		\bm{b}(t)
	\end{pmatrix}, 
	\]
	where
	\[
	\kappa(t) = \frac{\| \dot{\bm{\gamma}}(t) \times \ddot{\bm{\gamma}}(t) \|}{\| \dot{\bm{\gamma}}(t) \|^3}, \quad 
	\tau(t) = \frac{\det (\dot{\bm{\gamma}}(t), \ddot{\bm{\gamma}}(t), \dddot{\bm{\gamma}}(t))}{\| \dot{\bm{\gamma}}(t) \times \ddot{\bm{\gamma}}(t) \|^2}. 
	\]
	The functions $\kappa$ and $\tau$ are called the \textit{curvature} and \textit{torsion} of $\bm{\gamma}$. 
	It is well known that the curvature $\kappa$ and torsion $\tau$, regarded as functions of arc length, form a complete system of invariants of a Frenet curve under orientation-preserving Euclidean motions. 
	
	In this paper, we consider a generalization of the Frenet frame. 
	For this purpose, it is important to distinguish the principal normal vector from the direction of acceleration. 
	
	\begin{remark}\label{duality.of.frenet}
		We emphasize that, even for a Frenet curve, the vector $\bm{n}$ does not necessarily coincide with the direction of acceleration. 
		Indeed, a direct computation shows that 
		\[
		\ddot{\bm{\gamma}}(t) = \frac{d}{dt} \left( \| \dot{\bm{\gamma}}(t) \| \right) \bm{t}(t) + \| \dot{\bm{\gamma}}(t) \|^2 \kappa(t) \bm{n}(t). 
		\]
		Hence, $\ddot{\bm{\gamma}}$ is parallel to $\bm{n}$ if and only if $\|\dot{\bm{\gamma}}\|$ is constant. 
	\end{remark}
	
	We next examine the geometric structure underlying the Frenet frame. 
	The following observation serves as the starting point for our generalization. 
	
	\begin{remark}\label{frenet.characterization}
		The duality between $\bm{t}$ and $\bm{b}$ plays an essential role in the geometric structure of the Frenet frame (see Definition \ref{leg.curve}). 
		The pair $(\bm{t}, \bm{b})$ can be regarded as a Legendre curve in $\Delta$.
		When regarded as a Legendre curve in $\Delta$, it satisfies  
		\[
		\frac{d}{dt}
		\begin{pmatrix}
			\bm{t}(t) \\
			\bm{b}(t) \\
			-\bm{n}(t)
		\end{pmatrix}
		=
		\| \dot{\bm{\gamma}}(t) \|
		\begin{pmatrix}
			0 & 0 & -\kappa(t) \\
			0 & 0 & \tau(t) \\
			\kappa(t) & -\tau(t) & 0
		\end{pmatrix}
		\begin{pmatrix}
			\bm{t}(t) \\
			\bm{b}(t) \\
			-\bm{n}(t)
		\end{pmatrix}. 
		\]
		In particular, the Frenet frame is completely determined by the Legendre curve $(\bm{t} , \bm{b})$, since $-\bm{n} =\bm{t} \times \bm{b}$. 
		The skew-symmetric form of the coefficient matrix, together with its specific pattern of vanishing entries, reflects the essential geometric structure of the Frenet frame. 
		These relations can be illustrated schematically (see Figure \ref{fig:reg.frenet}). 
		
		Thus, the essential structure of the Frenet frame is encoded in the Legendre curve $(\bm t,\bm b)$ and the corresponding Frenet-Serret type formulas, rather than in the classical construction from higher-order derivatives of $\bm\gamma$. 
		Motivated by this observation, we introduce a generalized Frenet frame in the next section. 
	\end{remark}
	
	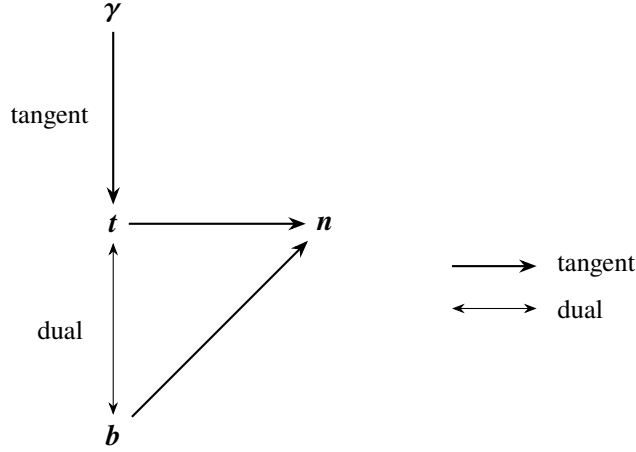
\begin{figure}[htbp]
		\centering
		\begin{tikzpicture}[
			>=Stealth,
			scale=1.4,
			tangent/.style={->, thick},
			dual/.style={<->, thin}
			]
			
			\node (gamma) at (0,4) {$\bm{\gamma}$};
			\node (t)     at (0,2) {$\bm{t}$};
			\node (b)     at (0,0) {$\bm{b}$};
			\node (n)     at (2,2) {$\bm{n}$};
			
			\draw[tangent] (gamma) -- (t);
			\draw[tangent] (t) -- (n);
			\draw[tangent] (b) -- (n);
			
			\draw[dual] (t) -- (b);
			
			\node at (-0.6,3) {\small tangent};
			\node at (-0.5,1) {\small dual};
			
			\draw[tangent] (3.2,1.6) -- (4.0,1.6);
			\node[anchor=west] at (4.1,1.6) {\small tangent};
			
			\draw[dual] (3.2,1.2) -- (4.0,1.2);
			\node[anchor=west] at (4.1,1.2) {\small dual};
			
		\end{tikzpicture}
		\caption{
			The dual structure underlying the Frenet frame. 
		}
		\label{fig:reg.frenet}
	\end{figure}
	
	\subsection{Framed curves}\label{framed}	
	To treat space curves with singularities in a unified manner, we briefly review the theory of framed curves. 
	For a detailed exposition, we refer to \cite{HT1}. 
	
	\begin{definition}
		We say that a map $(\bm{\gamma}, \bm{\nu}_1, \bm{\nu}_2):I \to \mathbb{R}^3 \times \Delta$ is a \textit{framed curve} if $\langle \dot{\bm{\gamma}}(t), \bm{\nu}_i(t) \rangle = 0$ for all $t \in I$ and $i = 1, 2$. 
		A curve $\bm{\gamma}: I \to \mathbb{R}^3$ is called a \textit{framed base curve} if there exists a map $(\bm{\nu}_1, \bm{\nu}_2): I \to \Delta$ such that $(\bm{\gamma}, \bm{\nu}_1, \bm{\nu}_2)$ is a framed curve. 
	\end{definition}
	
	Let $(\bm{\gamma}, \bm{\nu}_1, \bm{\nu}_2):I \to \mathbb{R}^3 \times \Delta$ be a framed curve. 
	We define $\bm{\mu}(t) = \bm{\nu}_1(t) \times \bm{\nu}_2(t)$. 
	Then $\{ \bm{\nu}_1, \bm{\nu}_2, \bm{\mu} \}$ forms an orthonormal frame along $\bm{\gamma}$. 
	\begin{proposition}[\cite{HT1}]
		Let $(\bm{\gamma}, \bm{\nu}_1, \bm{\nu}_2):I \to \mathbb{R}^3 \times \Delta$ be a framed curve. 
		Then the following Frenet-Serret type formulas hold: 
		\begin{align}\label{FS2}
			\frac{d}{dt}
			\begin{pmatrix}
				\bm{\nu}_1(t) \\
				\bm{\nu}_2(t) \\
				\bm{\mu}(t)
			\end{pmatrix}
			=
			\begin{pmatrix}
				0 & \ell(t) & m(t) \\
				-\ell(t) & 0 & n(t) \\
				-m(t) & -n(t) & 0
			\end{pmatrix}
			\begin{pmatrix}
				\bm{\nu}_1(t) \\
				\bm{\nu}_2(t) \\
				\bm{\mu}(t)
			\end{pmatrix}, \quad \dot{\bm{\gamma}}(t) = \alpha(t) \bm{\mu}(t), 
		\end{align}
		where $\ell(t) = \langle \dot{\bm{\nu}}_1(t), \bm{\nu}_2(t) \rangle$, $m(t) = \langle \dot{\bm{\nu}}_1(t), \bm{\mu}(t) \rangle$, $n(t) = \langle \dot{\bm{\nu}}_2(t), \bm{\mu}(t) \rangle$ and $\alpha(t) = \langle \dot{\bm{\gamma}}(t), \bm{\mu}(t) \rangle$. 
	\end{proposition}
	
	We call $(\ell, m, n, \alpha)$ the \textit{curvature} of the framed curve $(\bm{\gamma}, \bm{\nu}_1, \bm{\nu}_2)$. 
	The existence and uniqueness theorem for framed curves with prescribed curvature $(\ell, m, n, \alpha)$ was established in \cite{HT1}. 
	
	\begin{remark}
		In the definition of framed curves, the pair $(\bm{\nu}_1, \bm{\nu}_2)$ is not uniquely determined by $\bm{\gamma}$. 
		Indeed, $(\bm{\nu}_1, \bm{\nu}_2)$ forms an orthonormal basis of the (generalized) normal plane of $\bm{\gamma}$, and hence has freedom under rotations and reflections. 
	\end{remark}
	
	\begin{remark}\label{rem.analytic.framed}
		It was shown in \cite{HT1} that every real-analytic curve germ $\bm{\gamma}:(I,t_0) \to \mathbb{R}^3$ is a framed base curve. 
		Therefore, for real-analytic space curves, the theory of framed curves can always be applied locally. 
	\end{remark}
	
	If $\ell(t) = 0$ in (\ref{FS2}), then the frame $\{ \bm{\nu}_1, \bm{\nu}_2, \bm{\mu} \}$ is called a \textit{Bishop frame}. 
	By choosing a suitable rotation of $(\bm{\nu}_1, \bm{\nu}_2)$, a Bishop frame can always be constructed. 
	A Bishop frame is a distinguished frame, also known as a rotation-minimizing frame. 
	It is well-defined even for regular curves with vanishing curvature, and was originally introduced by Bishop \cite{B}. 
	
	We now describe an explicit construction of the Bishop frame. 
	For a framed curve $(\bm{\gamma}, \bm{\nu}_1, \bm{\nu}_2): I \to \mathbb{R}^3 \times \Delta$, we define $(\bm{b}_1, \bm{b}_2): I \to \Delta$ by 
	\[
	\begin{pmatrix}
		\bm{b}_1(t) \\
		\bm{b}_2(t)
	\end{pmatrix}
	=
	\begin{pmatrix}
		\cos \theta(t) & -\sin \theta(t) \\
		\sin \theta(t) & \cos \theta(t)
	\end{pmatrix}
	\begin{pmatrix}
		\bm{\nu}_1(t) \\
		\bm{\nu}_2(t)
	\end{pmatrix},
	\]
	where $\theta: I \to \mathbb{R}$ is a smooth function. 
	Then $(\bm{\gamma}, \bm{b}_1, \bm{b}_2)$ is also a framed curve. 
	The frame $\{ \bm{b}_1, \bm{b}_2, \bm{\mu} \}$ is obtained by rotating $\{ \bm{\nu}_1, \bm{\nu}_2, \bm{\mu} \}$ by $\theta$, and is called a \textit{rotated frame}. 
	By a direct computation, we have the Frenet-Serret type formulas: 
	\begin{align*}
		\frac{d}{dt}
		\begin{pmatrix}
			\bm{b}_1(t) \\
			\bm{b}_2(t) \\
			\bm{\mu}(t)
		\end{pmatrix}
		=
		\begin{pmatrix}
			0 & \overline{\ell}(t) & \overline{m}(t) \\
			-\overline{\ell}(t) & 0 & \overline{n}(t) \\
			-\overline{m}(t) & -\overline{n}(t) & 0
		\end{pmatrix}
		\begin{pmatrix}
			\bm{b}_1(t) \\
			\bm{b}_2(t) \\
			\bm{\mu}(t)
		\end{pmatrix}, \quad \dot{\bm{\gamma}}(t) = \alpha(t) \bm{\mu}(t), 
	\end{align*}
	where $\overline{\ell}(t) = \ell(t) - \dot{\theta}(t)$, $\overline{m}(t) = m(t) \cos \theta(t) - n(t) \sin \theta(t)$, and $\overline{n}(t) = m(t) \sin \theta(t) + n(t) \cos \theta(t)$. 
	If $\theta$ satisfies $\dot{\theta}(t) = \ell(t)$, then $\overline{\ell}(t) = 0$, and the frame $\{ \bm{b}_1, \bm{b}_2, \bm{\mu} \}$ is a \textit{Bishop frame} along $\bm{\gamma}$. 
	
	We conclude with two remarks on the Bishop frame. 

	\begin{remark}
		The Bishop frame admits the following geometric characterization. 
		Since $\overline{\ell}(t) = 0$, the derivatives of $\bm{b}_1$ and $\bm{b}_2$ have no components along $\bm{b}_1$ or $\bm{b}_2$, and hence are always parallel to $\bm{\mu}$. 
	\end{remark}
	
	\begin{remark}\label{remark.bishop2}
		A Bishop frame $\{ \bm{b}_1, \bm{b}_2, \bm{\mu} \}$ obtained by rotation is not uniquely determined by the original frame $\{ \bm{\nu}_1, \bm{\nu}_2, \bm{\mu} \}$. 
		Indeed, since $\theta$ is required only to satisfy $\dot{\theta}(t) = \ell(t)$, it is determined up to a constant of integration. 
	\end{remark}
	
	\section{Generalized Frenet frames}
	
	In this section, motivated by the geometric structure of the Frenet frame discussed in Remark \ref{frenet.characterization}, we introduce a generalized Frenet frame for singular space curves. 

	\begin{definition}\label{generalized.frenet}
		Let $\bm{\gamma}: I \to \mathbb{R}^3$ be a smooth curve. 
		We say that $\bm{\gamma}$ admits a (generalized) \textit{unit tangent vector} if there exist a function $\alpha : I \to \mathbb{R}$ and a map $\bm{t} : I \to S^2$ such that $\dot{\bm{\gamma}}(t) = \alpha(t)\bm{t}(t)$ for all $t \in I$. 
		In this case, $\alpha$ is called the (signed) \textit{speed} and $\bm{t}$ the (generalized) \textit{unit tangent vector}. 
		Furthermore, we say that $\bm{\gamma}$ admits a (generalized) \textit{Frenet frame} if $\bm{t}$ is a frontal in $S^2$. 
	\end{definition}
	
	Suppose that $\bm{\gamma}$ admits a (generalized) Frenet frame. 
	Let $\bm{t}$ be a (generalized) unit tangent vector of $\bm{\gamma}$, and let $\alpha$ be a corresponding (signed) speed. 
	Since $\bm{t}$ is a frontal in $S^2$, let $\bm{b}:I \to S^2$ be a dual of $\bm{t}$. 
	That is, $(\bm{t}, \bm{b})$ is a Legendre curve in $\Delta$. 
	We define $\bm{n}(t) = \bm{t}(t) \times \bm{b}(t)$. 
	Then $\{ \bm{t}, \bm{b}, \bm{n} \}$ forms an orthonormal frame along $\bm{\gamma}$. 
	By Proposition \ref{spherical.FS}, we obtain the following corollary. 
	
	\begin{corollary}\label{Frenet.Serret.G.Frenet}
		Under the above notation, the following Frenet-Serret type formulas hold: 
		\[
		\frac{d}{dt}
		\begin{pmatrix}
			\bm{t}(t) \\
			\bm{b}(t) \\
			\bm{n}(t)
		\end{pmatrix}
		=
		\begin{pmatrix}
			0 & 0 & \kappa(t) \\
			0 & 0 & \tau(t) \\
			-\kappa(t) & -\tau(t) & 0
		\end{pmatrix}
		\begin{pmatrix}
			\bm{t}(t) \\
			\bm{b}(t) \\
			\bm{n}(t)
		\end{pmatrix}, \quad \dot{\bm{\gamma}}(t) = \alpha(t) \bm{t}(t), 
		\]
		where $\kappa(t) = \langle \dot{\bm{t}}(t), \bm{n}(t) \rangle$ and $\tau(t) = \langle \dot{\bm{b}}(t), \bm{n}(t) \rangle$. 
	\end{corollary}
	
	In view of the duality between $\bm{t}$ and $\bm{b}$ and the structure of the coefficient matrix in the Frenet-Serret formulas for regular Frenet curves, both described in Remark \ref{frenet.characterization}, the frame $\{\bm{t}, \bm{b}, \bm{n}\}$ introduced in this section may be regarded as a generalization of the Frenet frame for regular curves. 
	
	Motivated by the above discussion, we call the frame $\{\bm{t}, \bm{b}, \bm{n}\}$ constructed in this section a (generalized) \textit{Frenet frame}. 
	We also call $\bm{b}$ the \textit{binormal vector}, $\bm{n}$ the \textit{principal normal vector}, $\kappa$ the \textit{curvature}, and $\tau$ the \textit{torsion}. 
	Note that $\kappa$ and $\tau$ are invariants in the sense of framed curves (see Section \ref{framed}). 
	
	\begin{remark}
		The case where the generalized unit tangent vector $\bm t$ is regular was studied in \cite{H2}. 
		Definition \ref{generalized.frenet} may be regarded as an extension of that framework to the singular case. 
	\end{remark}
	
	The following examples illustrate the scope of the notion of generalized Frenet frames. 
	We first show that every classical Frenet curve admits a generalized Frenet frame, then construct one for a singular curve, and finally present a regular curve that admits none. 
		
	\begin{example}
		Let $\bm{\gamma}:I \to \mathbb{R}^3$ be a Frenet curve with classical Frenet frame $\{\bm{t}, \bm{n}, \bm{b}\}$ (see Section \ref{sec.frenet}). 
		Since $\langle \bm{t}(t), \bm{b}(t) \rangle = \langle \dot{\bm{t}}(t), \bm{b}(t) \rangle = 0$ for all $t \in I$, $\bm{t}$ is a frontal in $S^2$. 
		Therefore, $\bm{\gamma}$ admits a generalized Frenet frame. 
	\end{example}
	
	\begin{example}\label{example.156}
		Consider the curve $\bm{\gamma}: I \to \mathbb{R}^3$ defined by
		\[
		\bm{\gamma}(t)= \left( \frac{1}{2} t^2, \frac{1}{5} t^5, \frac{1}{6} t^6 \right). 
		\]
		By direct computation,
		\[
		\dot{\bm{\gamma}}(t) = \left( t, t^4, t^5 \right) = t \sqrt{1 + t^6 + t^8} \cdot \frac{\left( 1, t^3, t^4 \right)}{\sqrt{1 + t^6 + t^8}}. 
		\]
		Hence, setting
		\[
		\alpha(t) = t \sqrt{1 + t^6 + t^8}, \quad \bm{t}(t) = \frac{\left( 1, t^3, t^4 \right)}{\sqrt{1 + t^6 + t^8}}, 
		\]
		we obtain $\dot{\bm{\gamma}}(t)=\alpha(t)\bm{t}(t)$, where $\alpha$ is the speed and $\bm{t}$ is the unit tangent vector of $\bm{\gamma}$. 
		By further direct computation, we obtain 
		\begin{align*}
			\dot{\bm{t}}(t) &= \frac{1}{(1 + t^6 + t^8)^\frac{3}{2}} \left( - 3t^5 - 4t^7, 3t^2 - t^{10}, 4t^3 + t^9 \right) \\
			&= \frac{t^2 \sqrt{9 + 16t^2 + 9t^6 + 26t^8 + 16t^{10} + t^{14} + t^{16}}}{(1 + t^6 + t^8)^\frac{3}{2}} \cdot \frac{\left( -3t^3 - 4t^5, 3 - t^8, 4t + t^7 \right)}{\sqrt{9 + 16t^2 + 9t^6 + 26t^8 + 16t^{10} + t^{14} + t^{16}}}. 
		\end{align*}
		Next, define 
		\begin{align*}
		\bm{b}(t) &= \frac{\left( -3t^3 - 4t^5, 3 - t^8, 4t + t^7 \right)}{\sqrt{9 + 16t^2 + 9t^6 + 26t^8 + 16t^{10} + t^{14} + t^{16}}} \times \bm{t}(t) \\
		&= \frac{\left( -t^4 - t^{10} - t^{12}, 4t+ 4t^7 + 4t^9, -3 - 3t^6 -3t^8 \right)}{\sqrt{1 + t^6 + t^8} \sqrt{9 + 16t^2 + 9t^6 + 26t^8 + 16t^{10} + t^{14} + t^{16}}}. 
		\end{align*}
		Then, $(\bm{t}, \bm{b})$ is a Legendre curve in $\Delta$, and $\bm{b}$ is the binormal vector of $\bm{\gamma}$. 
		By routine computations, we obtain
		\begin{align*}
			\bm{n}(t) &= \bm{t}(t) \times \bm{b}(t) = \frac{\left( -3t^3 - 4t^5, 3 - t^8, 4t + t^7 \right)}{\sqrt{9 + 16t^2 + 9t^6 + 26t^8 + 16t^{10} + t^{14} + t^{16}}}, \\
			\kappa(t) &= \langle \dot{\bm{t}}(t), \bm{n}(t) \rangle = \frac{t^2 \sqrt{9 + 16t^2 + 9t^6 + 26t^8 + 16t^{10} + t^{14} + t^{16}}}{(1 + t^6 + t^8)^\frac{3}{2}}, \\
			\tau(t) &= \langle \dot{\bm{b}}(t), \bm{n}(t) \rangle = \frac{12 \sqrt{1 + t^6 + t^8}}{9 + 16t^2  + t^8}, 
		\end{align*}
		where $\bm{n}$ is the principal normal vector, $\kappa$ is the curvature, and $\tau$ is the torsion. 
		We note that both $\bm{\gamma}$ and $\bm{t}$ are singular at $t=0$ (see Figure \ref{fig:example}). 
		
		\begin{figure}[htbp]
			\centering
			\includegraphics[width=.75\textwidth]{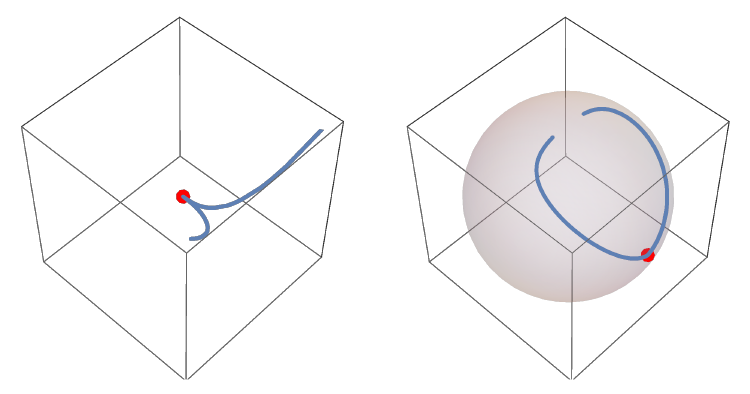}
			\caption{From left to right: the curve $\bm{\gamma}$ and its tangent vector $\bm{t}$ in Example~\ref{example.156}.}
			\label{fig:example}
		\end{figure}
	\end{example}
	
	\begin{example}\label{example.flat}
		Consider the curve $\bm{\gamma}: \mathbb{R} \to \mathbb{R}^3$ defined by
		\[
		\bm{\gamma}(t) =
		\begin{cases}
			\left( t, e^{-\frac{1}{t^2}}, 0 \right) & (t > 0), \\
			\left( 0, 0, 0 \right) & (t = 0), \\
			\left( t, 0, e^{-\frac{1}{t^2}} \right) & (t < 0).
		\end{cases}
		\]
		Then $\bm{\gamma}$ is a smooth regular curve on $\mathbb{R}$. 
		Although its unit tangent vector $\bm{t}$ is smooth on $\mathbb{R}$, it does not admit a smooth dual across $t=0$ (see Figure \ref{fig:example2}). 
		Hence, $\bm{t}$ is not a frontal in $S^2$, and consequently $\bm{\gamma}$ does not admit a generalized Frenet frame on the whole of $\mathbb{R}$. 
		
		\begin{figure}[htbp]
			\centering
			\includegraphics[width=.75\textwidth]{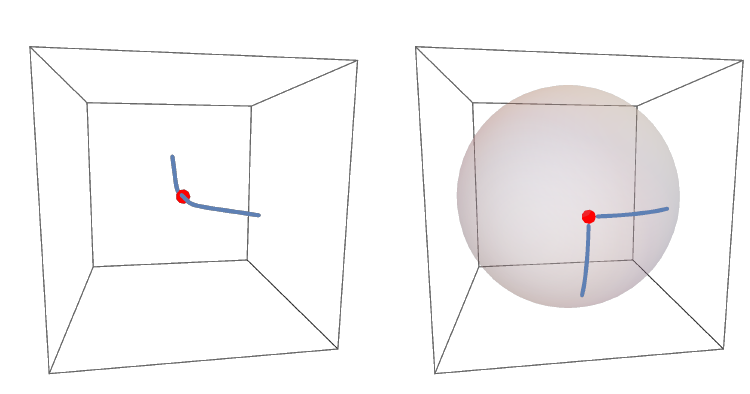}
			\caption{From left to right: the curve $\bm{\gamma}$ and its tangent vector $\bm{t}$ in Example~\ref{example.flat}.}
			\label{fig:example2}
		\end{figure}
	\end{example}
	
	\begin{remark}
		Note that Definition \ref{generalized.frenet} allows degenerate cases such as $\bm{\gamma}(t) = (0,0,0)$, which admit a Frenet frame in this generalized sense. 
		For applications, it is therefore necessary to impose additional conditions, such as the density of regular points. 
	\end{remark}
	
	\section{Frame sequence}
	We consider the (generalized) Frenet frame $\{\bm{t}, \bm{b}, \bm{n}\}$. 
	By definition, $\bm{t}$ and $\bm{b}$ are frontals in $S^2$. 
	We now investigate the case where $\bm{n}$ is also a frontal. 
	For convenience, we introduce the notation $\bm{t}_0 = \bm{t}$, $\bm{d}_0 = \bm{b}$, and $\bm{t}_1 = \bm{n}$. 	
	Let $\bm{d}_1:I \to S^2$ be a dual of $\bm{t}_1$, so that $(\bm{t}_1, \bm{d}_1)$ is a Legendre curve in $\Delta$. 
	We define $\bm{t}_2(t) = \bm{t}_1(t) \times \bm{d}_1(t)$. 
	Then $\{ \bm{t}_1, \bm{d}_1, \bm{t}_2 \}$ forms an orthonormal frame along $\bm{\gamma}$. 
	This construction is illustrated in Figure \ref{fig:frame.sequence1}. 
	
	\begin{figure}[htbp]
		\centering
		\begin{tikzpicture}[
			>=Stealth,
			scale=1.4,
			tangent/.style={->, thick},
			dual/.style={<->, thin}
			]
			
			\node (gamma) at (0,4) {$\bm{\gamma}$};
			\node (t0)     at (0,2) {$\bm{t}_0$};
			\node (d0)     at (0,0) {$\bm{d}_0$};
			\node (t1)     at (2,2) {$\bm{t}_1$};
			\node (d1)     at (2,0) {$\bm{d}_1$};      
			\node (t2)     at (4,2) {$\bm{t}_2$};      
			
			\draw[tangent] (gamma) -- (t0);
			\draw[tangent] (t0) -- (t1);
			\draw[tangent] (d0) -- (t1);
			\draw[tangent] (t1) -- (t2);   
			\draw[tangent] (d1) -- (t2);
			
			\draw[dual] (t0) -- (d0);
			\draw[dual] (t1) -- (d1);      
			
			\node at (-0.6,3) {\small tangent};
			\node at (-0.5,1) {\small dual};
			
			\draw[tangent] (5.2,1.6) -- (6.0,1.6);
			\node[anchor=west] at (6.1,1.6) {\small tangent};
			
			\draw[dual] (5.2,1.2) -- (6.0,1.2);
			\node[anchor=west] at (6.1,1.2) {\small dual};
			
		\end{tikzpicture}
		\caption{
			Relationships among $\bm{t}_i$, $\bm{d}_i$ ($i = 0,1$) and $\bm{t}_2$. 
		}
		\label{fig:frame.sequence1}
	\end{figure}
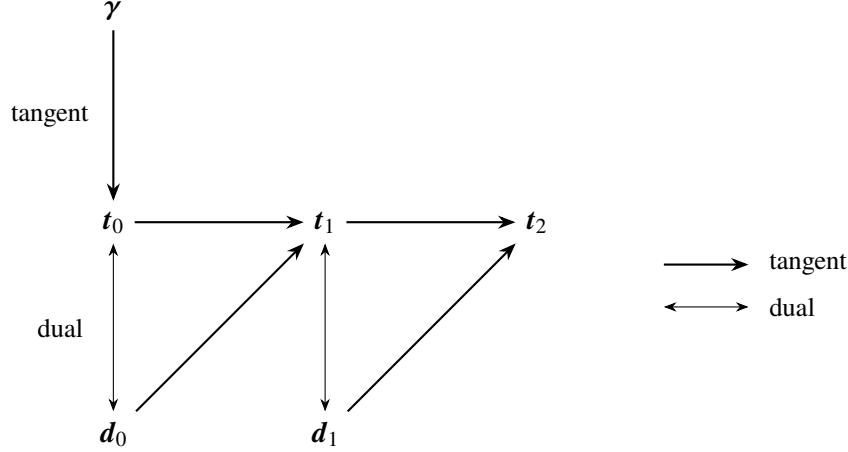
	We call $\{ \bm{t}_0, \bm{d}_0, \bm{t}_1 \}$ the \textit{$0$th-order Frenet frame}, and $\{ \bm{t}_1, \bm{d}_1, \bm{t}_2 \}$ the \textit{$1$st-order Frenet frame}. 
	Inductively, whenever $\bm{t}_k$ is a frontal in $S^2$, we choose a dual $\bm{d}_k$ of $\bm{t}_k$, define $\bm{t}_{k+1}(t) = \bm{t}_k(t) \times \bm{d}_k(t)$, and call $\{ \bm{t}_k, \bm{d}_k, \bm{t}_{k+1} \}$ the \textit{$k$th-order Frenet frame} along $\bm{\gamma}$. 
	
	On the other hand, starting from $\{ \bm{t}_0, \bm{d}_0, \bm{t}_1 \}$, we construct a Bishop frame along $\bm{\gamma}$ by fixing $t_0$ and rotating $\bm{d}_0$ and $\bm{t}_1$ (see Remark \ref{remark.bishop2}). 
	We denote the resulting Bishop frame  by $\{ \bm{t}_{-1}[t_0], \bm{d}_{-1}[t_0], \bm{t}_0 \}$. 
	This construction is illustrated in Figure \ref{fig:frame.sequence2}. 
	
	\begin{figure}[htbp]
		\centering
		\begin{tikzpicture}[
			>=Stealth,
			scale=1.4,
			tangent/.style={->, thick},
			dual/.style={<->, thin}
			]
			
			\node (gamma) at (0,4) {$\bm{\gamma}$};
			
			\node (t0)  at (0,2) {$\bm{t}_0$};
			\node (d0)  at (0,0) {$\bm{d}_0$};
			
			\node (t1)  at (2,2) {$\bm{t}_1$};
			\node (d1)  at (2,0) {$\bm{d}_1$};
			
			\node (t2)  at (4,2) {$\bm{t}_2$};
			
			\node (t-1) at (-2,2) {$\bm{t}_{-1}[t_0]$};
			\node (d-1) at (-2,0) {$\bm{d}_{-1}[t_0]$};
			
			\draw[tangent] (gamma) -- (t0);
			\draw[tangent] (t0) -- (t1);
			\draw[tangent] (d0) -- (t1);
			\draw[tangent] (t1) -- (t2);
			\draw[tangent] (d1) -- (t2);
			\draw[tangent] (t-1) -- (t0);
			\draw[tangent] (d-1) -- (t0);
			
			\draw[dual] (t0) -- (d0);
			\draw[dual] (t1) -- (d1);
			\draw[dual] (t-1) -- (d-1);
			
			\node at (-0.6,3) {\small tangent};
			\node at (-2.5,1) {\small dual};
			
			\draw[tangent] (5.2,1.6) -- (6.0,1.6);
			\node[anchor=west] at (6.1,1.6) {\small tangent};
			
			\draw[dual] (5.2,1.2) -- (6.0,1.2);
			\node[anchor=west] at (6.1,1.2) {\small dual};
			
		\end{tikzpicture}
		\caption{
			Relationships among $\bm{t}_i$, $\bm{d}_i$ ($i = -1,0,1$) and $\bm{t}_2$.
		}
		\label{fig:frame.sequence2}
	\end{figure}
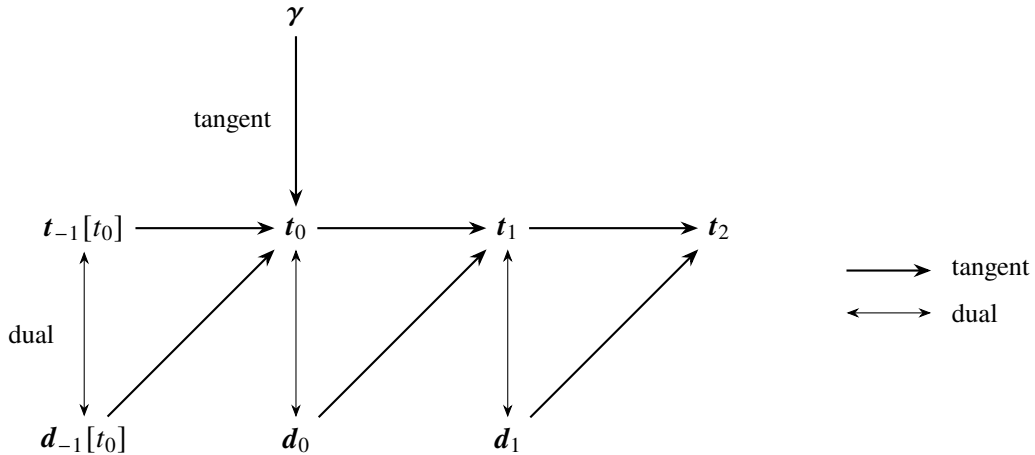
	
	We call $\{ \bm{t}_{-1}[t_0], \bm{d}_{-1}[t_0], \bm{t}_0 \}$ the \textit{$(-1)$st-order Frenet frame}. 
	By analogy with the construction of the Bishop frame, we fix $\bm{t}_k$ and rotate $\bm{d}_k$ and $\bm{t}_{k+1}$
	to obtain $\bm{t}_{k-1}$ and $\bm{d}_{k-1}$. 
	Therefore, for every negative integer $k$, $\{ \bm{t}_k, \bm{d}_k, \bm{t}_{k+1} \}$ can be constructed recursively. 
	
	In summary, whenever the above construction can be continued in both directions, we obtain a $k$th-order Frenet frame along $\bm{\gamma}$ for each $k \in \mathbb{Z}$. 
	
	\begin{definition}
		A bi-infinite sequence $\{(\bm{t}_k,\bm{d}_k,\bm{t}_{k+1})\}_{k\in\mathbb Z}$ 
		constructed as above is called a \textit{frame sequence} associated with $\bm{\gamma}$. 
		For each $k\in\mathbb Z$, we call $\bm{t}_k$ the \textit{$k$th-order tangent vector} and $\bm{d}_k$ the \textit{$k$th-order dual vector}. 
	\end{definition}
	
	The construction of the $k$th-order Frenet frame along $\bm{\gamma}$ admits a natural interpretation in the framework of frontals in $S^2$. 
	Starting from the original Frenet frame (the $0$th order frame), the $k$th order frames are obtained through an iterative process involving the evolutes and involutes of frontals on $S^2$. 
	We summarize this situation in Figure \ref{fig:frame.sequence3}. 
	
		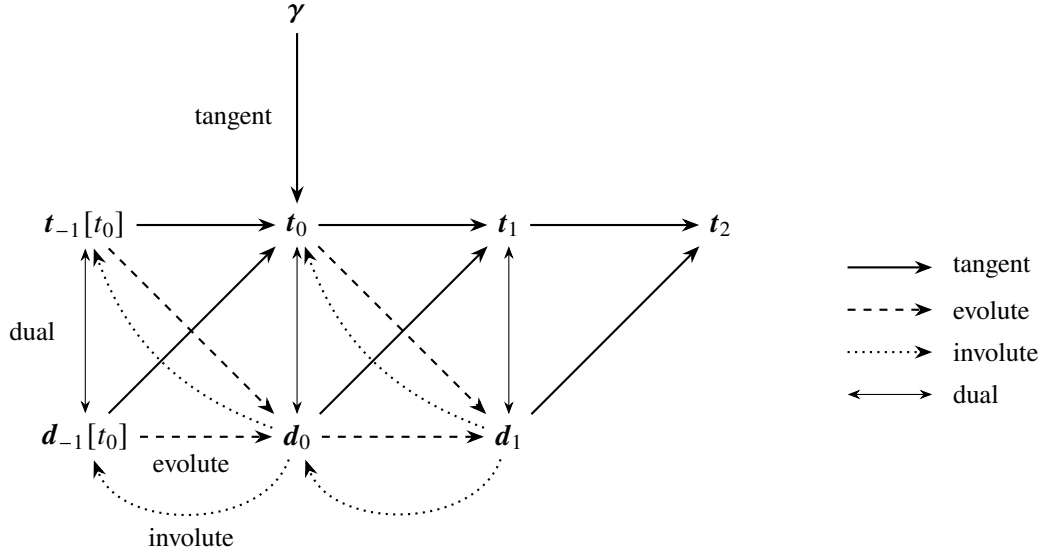
\begin{figure}[htbp]
		\centering
		\begin{tikzpicture}[
			>=Stealth,
			scale=1.4,
			tangent/.style={->, thick},
			evolute/.style={->, dashed, thick},
			involute/.style={->, dotted, thick},
			dual/.style={<->, thin}
			]
			
			\node (gamma) at (0,4) {$\bm{\gamma}$};
			
			\node (t0)  at (0,2) {$\bm{t}_0$};
			\node (d0)  at (0,0) {$\bm{d}_0$};
			
			\node (t1)  at (2,2) {$\bm{t}_1$};
			\node (d1)  at (2,0) {$\bm{d}_1$};
			
			\node (t2)  at (4,2) {$\bm{t}_2$};
			
			\node (t-1) at (-2,2) {$\bm{t}_{-1}[t_0]$};
			\node (d-1) at (-2,0) {$\bm{d}_{-1}[t_0]$};
			
			\draw[tangent] (gamma) -- (t0);
			\draw[tangent] (t0) -- (t1);
			\draw[tangent] (d0) -- (t1);
			\draw[tangent] (t1) -- (t2);
			\draw[tangent] (d1) -- (t2);
			\draw[tangent] (t-1) -- (t0);
			\draw[tangent] (d-1) -- (t0);
			
			\draw[dual] (t0) -- (d0);
			\draw[dual] (t1) -- (d1);
			\draw[dual] (t-1) -- (d-1);
			
			\draw[evolute] (t-1) -- (d0);
			\draw[evolute] (d-1) -- (d0);
			\draw[evolute] (t0) -- (d1);
			\draw[evolute] (d0) -- (d1);
			
			\draw[involute, bend left=25] (d0) to (t-1);
			\draw[involute] (d0) to[out=250, in=290] (d-1);
			\draw[involute, bend left=25] (d1) to (t0);
			\draw[involute] (d1) to[out=250, in=290] (d0);
			
			\node at (-0.6,3) {\small tangent};
			\node at (-2.5,1) {\small dual};
			\node at (-1,-0.25) {\small evolute};
			\node at (-1,-0.95) {\small involute};
			
			\draw[tangent] (5.2,1.6) -- (6.0,1.6);
			\node[anchor=west] at (6.1,1.6) {\small tangent};
			
			\draw[evolute] (5.2,1.2) -- (6.0,1.2);
			\node[anchor=west] at (6.1,1.2) {\small evolute};
			
			\draw[involute] (5.2,0.8) -- (6.0,0.8);
			\node[anchor=west] at (6.1,0.8) {\small involute};
			
			\draw[dual] (5.2,0.4) -- (6.0,0.4);
			\node[anchor=west] at (6.1,0.4) {\small dual};
			
		\end{tikzpicture}
		\caption{
			Relationships among $\bm{t}_i$, $\bm{d}_i$ ($i = -1,0,1$) and $\bm{t}_2$. 
		}
		\label{fig:frame.sequence3}
	\end{figure}
	
	The following theorem follows from Proposition \ref{spherical.FS} and Corollary \ref{Frenet.Serret.G.Frenet}. 
	
	\begin{theorem}[The $k$th-order Frenet-Serret formulas]
		For the $k$th-order Frenet frame along $\bm{\gamma}$, the following Frenet-Serret type formulas hold: 
		\begin{align}\label{kGFrenet}
			\frac{d}{dt}
			\begin{pmatrix}
				\bm{t}_k(t) \\
				\bm{d}_k(t) \\
				\bm{t}_{k+1}(t)
			\end{pmatrix}
			=
			\begin{pmatrix}
				0 & 0 & \kappa_k(t) \\
				0 & 0 & \tau_k(t) \\
				-\kappa_k(t) & -\tau_k(t) & 0
			\end{pmatrix}
			\begin{pmatrix}
				\bm{t}_k(t) \\
				\bm{d}_k(t) \\
				\bm{t}_{k+1}(t)
			\end{pmatrix}, 
		\end{align}
		where $\kappa_k(t) = \langle \dot{\bm{t}}_k(t), \bm{t}_{k+1}(t) \rangle$ and $\tau_k(t) = \langle \dot{\bm{d}}_k(t), \bm{t}_{k+1}(t) \rangle$. 
	\end{theorem}
	
	We call (\ref{kGFrenet}) the \textit{$k$th-order Frenet-Serret formulas}, $\kappa_k$ the \textit{$k$th-order curvature}, and $\tau_k$ the \textit{$k$th-order torsion}. 
	
	\begin{remark}
		For a Frenet curve, the 0th-order Frenet frame corresponds to the classical Frenet frame. 
		On the other hand, the $(-1)$st-order Frenet frame corresponds to a Bishop frame. 
	\end{remark}
	
	\begin{remark}
		Similar higher-order invariants and hierarchies of invariants have appeared in the study of ruled and developable surfaces, see \cite{H1, IST1}. 
		The main difference is that our approach focuses not only on the invariants themselves but also on the recursive structure of the corresponding frames. 
	\end{remark}
	
	To better understand the recursive structure of the frame sequence, we consider the case where the pair $(\kappa_k(t),\tau_k(t))$ does not vanish simultaneously. 
	The following theorem gives recursive formulas relating the $k$th-order and $(k+1)$st-order curvatures and torsions. 
	
	\begin{theorem}\label{thm:recursive}
		Assume that $(\kappa_k(t), \tau_k(t)) \neq (0,0)$ for all $t\in I$. 
		Then $(\bm{t}_{k+1}, \bm{d}_{k+1})$ is a Legendre curve in $\Delta$, where 		
		\[
		\bm{d}_{k+1}(t) = \frac{- \tau_k(t) \bm{t}_k(t) + \kappa_k(t) \bm{d}_k(t)}{\sqrt{\kappa_k^2(t) + \tau_k^2(t)}}. 
		\]
		Moreover, the $(k+1)$st-order curvature and torsion are
		\begin{align}\label{eq.recursive.formula}
			\kappa_{k+1}(t) = \sqrt{\kappa_k^2(t)+\tau_k^2(t)}, \quad
			\tau_{k+1}(t) = \frac{\kappa_k(t)\dot{\tau}_k(t)-\dot{\kappa}_k(t)\tau_k(t)}{\kappa_k^2(t)+\tau_k^2(t)}. 
		\end{align}
	\end{theorem}
	
	\begin{proof}
		By the $k$th-order Frenet-Serret formulas, we have $\dot{\bm{t}}_{k+1}(t) = -\kappa_k(t)\bm{t}_k(t)-\tau_k(t)\bm{d}_k(t)$. 
		By the definition of $\bm{d}_{k+1}$, we have $\| \bm{d}_{k+1}(t) \| =1$ and $\langle \bm{d}_{k+1}(t), \bm{t}_{k+1}(t) \rangle = \langle \bm{d}_{k+1}(t), \dot{\bm{t}}_{k+1}(t) \rangle = 0$. 
		Hence, $(\bm{t}_{k+1}, \bm{d}_{k+1})$ is a Legendre curve in $\Delta$. 
		A straightforward computation yields (\ref{eq.recursive.formula}). 
	\end{proof}
	
	\begin{corollary}\label{coro.recursive}
		Under the assumptions of Theorem \ref{thm:recursive}, let 
		\[
		(\kappa_k(t), \tau_k(t)) = (r_k(t)\cos\theta_k(t), r_k(t)\sin\theta_k(t)), 
		\]
		where $r_k$ and $\theta_k$ denote the norm and argument of the vector $(\kappa_k(t),\tau_k(t))$, respectively.
		Then
		\[
		\kappa_{k+1}(t)=r_k(t),
		\qquad
		\tau_{k+1}(t)=\dot{\theta}_k(t).
		\]
	\end{corollary}

	Corollary \ref{coro.recursive} shows that $\kappa_{k+1}$ is the norm of the vector $(\kappa_k,\tau_k)$, whereas $\tau_{k+1}$ is the derivative of its argument. 
	Hence, the $(k+1)$st-order curvature and torsion measure, respectively, the magnitude and angular variation of the preceding invariants. 
	
	\begin{remark}
		By Theorem \ref{thm:recursive}, if there exists an integer $k$ such that $(\kappa_k(t), \tau_k(t)) \neq (0,0)$ for all $t\in I$, then the $j$th-order Frenet frame is well defined for every $j \ge k$. 
		In particular, if $\bm{\gamma}$ is a Frenet curve, then $\kappa_0(t) \neq 0$ for all $t\in I$. 
		Hence the $k$th-order Frenet frame is well defined for every integer $k$. 
	\end{remark}
	
	\begin{remark}
		By Remarks \ref{rem.analytic.frontal} and \ref{rem.analytic.framed}, in the real-analytic category, $\bm{\gamma}$ and $\bm{t}_k \ (k \in \mathbb{Z})$ are locally a framed base curve and a spherical frontal, respectively. 
		Hence the existence of a frame sequence is guaranteed locally. 
	\end{remark}
	
	\begin{example}\label{0th.helix}
		Consider the helix $\bm{\gamma}: \mathbb{R} \to \mathbb{R}^3$ defined by $\bm{\gamma}(t) = \left( a \cos t, a \sin t, bt \right)$, where $a$ and $b$ are positive constants. 
		By direct computation, 
		\[
		\dot{\bm{\gamma}}(t) = \left( -a \sin t, a \cos t, b \right) = \sqrt{a^2 + b^2} \cdot \frac{\left( -a \sin t, a \cos t, b \right)}{\sqrt{a^2 + b^2}}. 
		\]
		Hence, setting 
		\[
		\alpha(t) = \sqrt{a^2 + b^2}, \quad \bm{t}_0(t) = \frac{\left( -a \sin t, a \cos t, b \right)}{\sqrt{a^2 + b^2}}, 
		\]
		we obtain $\dot{\bm{\gamma}}(t) = \alpha(t) \bm{t}_0(t)$, where $\alpha$ is the speed and $\bm{t}_0$ is the 0th-order tangent (the tangent) vector of $\bm{\gamma}$. 
		By further direct computation, we obtain 
		\[
		\dot{\bm{t}}_0(t) = \frac{\left( -a \cos t, -a \sin t, 0 \right)}{\sqrt{a^2 + b^2}} = \frac{a}{\sqrt{a^2 + b^2}} \cdot \frac{\left( -a \cos t, -a \sin t, 0 \right)}{a}. 
		\]
		Next, define
		\[
		\bm{d}_0(t) = \frac{\left( -a \cos t, -a \sin t, 0 \right)}{a} \times \bm{t}_0(t) = \frac{\left( -b \sin t, b \cos t, -a \right)}{\sqrt{a^2 + b^2}}. 
		\]
		Then, $(\bm{t}_0, \bm{d}_0)$ is a Legendre curve in $\Delta$, and $\bm{d}_0$ is the 0th-order dual (the binormal) vector of $\bm{\gamma}$. 
		By routine computations, 
		\begin{align*}
			\bm{t}_1(t) &= \bm{t}_0(t) \times \bm{d}_0(t) = \frac{\left( -a \cos t, -a \sin t, 0 \right)}{a}, \\
			\kappa_0(t) &= \langle \dot{\bm{t}}_0(t), \bm{t}_1(t) \rangle = \frac{a}{\sqrt{a^2 + b^2}}, \\
			\tau_0(t) &= \langle \dot{\bm{d}}_0(t), \bm{t}_1(t) \rangle = \frac{b}{\sqrt{a^2 + b^2}}, 
		\end{align*}
		where $\bm{t}_1$ is the 1st-order tangent (the principal normal) vector, $\kappa_0$ is 0th-order curvature, and $\tau_0$ is the 0th-order torsion. 
		By repeatedly applying Theorem \ref{thm:recursive}, we obtain
		\[
		\kappa_k(t)=1, \quad \tau_k(t)=0
		\]
		for all $k \ge 1$. 
		Hence, after the first step, the frame sequence reaches a fixed point in the sense that all higher-order curvatures and torsions remain unchanged. 
	\end{example}
	
	\begin{example}\label{1st.helix}
		Consider the curve $\bm{\gamma}: \mathbb{R} \to \mathbb{R}^3$ defined by 
		\begin{align*}
			\bm{\gamma}(t) =& \biggl( - \frac{a^2 - b^2}{2a} \left( \frac{\cos((a + b) t)}{(a + b)^2} + \frac{\cos((a - b) t)}{(a - b)^2} \right), \\
			& \qquad - \frac{a^2 - b^2}{2a} \left( \frac{\sin((a + b) t)}{(a + b)^2} + \frac{\sin((a - b) t)}{(a - b)^2} \right), - \frac{\sqrt{a^2 - b^2}}{ab} \cos(bt) \biggr), 
		\end{align*}
		where $a$ and $b$ are positive constants satisfying $a>b$. 
		Since $\| \dot{\bm{\gamma}}(t) \| = 1$, setting $\alpha(t) = 1$ and $\bm{t}_0(t) = \dot{\bm{\gamma}}(t)$,  we obtain $\dot{\bm{\gamma}}(t) = \alpha(t) \bm{t}_0(t)$, where $\alpha$ is the speed and $\bm{t}_0$ is the 0th-order tangent (the tangent) vector of $\bm{\gamma}$. 
		By direct computation with a suitable generalized Frenet frame, we obtain
		\[
		\kappa_0(t) = \sqrt{a^2 - b^2} \cos(bt), \quad \tau_0(t) = \sqrt{a^2 - b^2} \sin(bt). 
		\]
		By repeatedly applying Theorem \ref{thm:recursive}, we obtain
		\[
		\kappa_1(t) = \sqrt{a^2 - b^2}, \quad \tau_1(t) = b, \quad \kappa_k(t) = a, \quad \tau_k(t) = 0
		\]
		for all $k \geq 2$. 
		This curve was introduced in \cite{IT1} as an example of a slant helix (1st-order helix). 
	\end{example}
	
	In \cite{IST1}, the notion of $k$th-order helices was introduced as a generalization of the classical helix, along with invariants defined in terms of the classical curvature and torsion that characterize such helices. 
	Example \ref{0th.helix} is a $0$th-order helix, while Example \ref{1st.helix} is a $1$st-order helix. 
	The above computations show that the higher-order curvatures and torsions naturally detect the order of these helices. 
	Indeed, the corresponding invariants become constant precisely at the relevant level. 
	This indicates that the higher-order curvatures and torsions are natural invariants for the study of $k$th-order helices. 
	
	The frame sequence also admits an interpretation from the viewpoint of the Taylor expansion of a curve. 
	Bouquet's formula can be rewritten in terms of the frame sequence. 
	
	\begin{proposition}
		Let $\bm{\gamma}$ be a space curve admitting a frame sequence. 
		Then, 
		\begin{align*}
			\bm{\gamma}(t) &= \bm{\gamma}(t_0) + (t - t_0) \alpha(t_0) \bm{t}_0(t_0) + \frac{1}{2!} (t - t_0)^2  \left( \dot{\alpha}(t_0) \bm{t}_0(t_0) + \alpha(t_0) \kappa_0(t_0) \bm{t}_1(t_0) \right) \\
			& \qquad + \frac{1}{3!} (t - t_0)^3 \left( \ddot{\alpha}(t_0) \bm{t}_0(t_0) + \left( 2 \dot{\alpha}(t_0) \kappa_0(t_0) + \alpha(t_0) \dot{\kappa}_0(t_0)  \right) \bm{t}_1(t_0) + \alpha(t_0) \kappa_0(t_0) \kappa_1(t_0) \bm{t}_2(t_0) \right) \\
			& \qquad + O((t - t_0)^4), 
		\end{align*}
		where $O((t - t_0)^4)$ denotes a term of order at least four. 
	\end{proposition}
	\begin{proof}
		The result follows from direct computations and the Taylor expansion of $\bm{\gamma}$ at $t_0$. 
	\end{proof}
	
	\section*{Acknowledgements}
	This work was supported by JSPS KAKENHI Grant Number JP25K17257.

	\bigskip
	
	\noindent
	Faculty of Science and Technology, \\
	Chitose Institute of Science and Technology, \\
	Chitose 066-8655, Japan
	
	\medskip
	
	\noindent
	E-mail address: s-honda@photon.chitose.ac.jp
	
\end{document}